\documentclass[11pt]{article}
\usepackage{amsfonts,amsmath,amssymb,amsthm}
\usepackage{latexsym}
\usepackage{algorithm}
\usepackage{xcolor}
\usepackage{dsfont}
\usepackage{abstract}
\usepackage{graphicx}
\usepackage{multirow}
\usepackage[margin=1.00in]{geometry}
\usepackage{enumitem}
\usepackage{url}

\definecolor{pred}{RGB}{148,55,61}
\usepackage[colorlinks,linkcolor=blue,citecolor=blue,anchorcolor=blue]{hyperref}

\definecolor{blue}{rgb}{0,0,0.9}
\definecolor{red}{rgb}{0.9,0,0}
\definecolor{green}{rgb}{0,0.9,0}

\def\Re{\mathbb{R}}

\definecolor{red}{rgb}{0.5,0,0}
\small\normalsize
\newtheorem{theorem}{Theorem}[section]
\newtheorem{proposition}{Proposition}[section]
\newtheorem{lemma}{Lemma}[section]

\newtheorem{remark}{Remark}[section]

\title{A dual based semismooth Newton method for a class of sparse Tikhonov regularization\footnote{Submitted to the editors 2019-04-13, revised: 2020-01-15.}}
\date{December 10, 2020}
\author{{Ning Zhang}\thanks{
School of Computer Science and Technology, Dongguan University of Technology, Dongguan 523808, China. The research of this author was supported by the National Natural Science Foundation of China (11901083~\&11801023) (\tt{zhangning@dgut.edu.cn}) }}

\begin{document}
\maketitle
\vspace{2mm}
\begin{center}
\parbox{13.5cm}{\small
\textbf{Abstract.} It is well known that Tikhonov regularization is one of the most commonly used methods for solving the ill-posed problems.
{One of the most widely applied approaches is based on constructing a new dataset whose sample size is greater than the original one. The enlarged sample size may bring additional computational difficulties.
In this paper, we aim to make full use of Tikhonov regularization and develop a dual based semismooth Newton (DSSN) method without destroying the structure of dataset.}
From the point of view of theory, we will show that {the DSSN method is a globally convergent method with at least R-superlinear rate of convergence.}
In the numerical computation aspect, we evaluate the performance of the DSSN method by solving a class of sparse Tikhonov regularization with high-dimensional datasets.
\\[2mm]
}		
\end{center}



\section{Introduction}
In order to overcome the drawbacks of the ordinary Least squares (OLS) in prediction accuracy and interpretation,
Tibshirani introduced the Lasso technique \cite{tibshirani1996regression}, which is essentially an $\ell_1$ norm regularized least-square problem.
With the emergence of a large number of  low sample size and high dimensional data, Lasso and its variants have drawn more and more attentions.
In this paper, we consider the sparse Tikhonov regularization problem in the following form
\begin{equation}\label{RLasso}
\min\limits_{x\in\Re^n}~ f(x):=\frac{1}{2}\|Ax-b\|^2+\frac{\lambda}{2}\|x\|^2+\varphi(x),
\end{equation}
where $\lambda> 0$,  $A\in\Re^{m\times n}$ {is a given data matrix}, $\varphi:\Re^n\rightarrow (-\infty,+\infty]$
is a closed proper convex (not necessarily smooth) function.
If $\varphi(x):=\mu\|x\|_1,\,\,\mu>0$, problem \eqref{RLasso} reduces to the following elastic net regularized regression, which was proposed by Zou and Hastie \cite{zou2005regularization} to resolve the limitations of Lasso,
\begin{equation}\label{model-Elass}
\min\limits_{x\in\Re^n}~ \frac{1}{2}\|Ax-b\|^2+\frac{\lambda}{2}\|x\|^2+\mu\|x\|_1.
\end{equation}
The above model has been extensively used in many fields, such as uncovering the consistent networks of functional disconnection in Alzheimer's disease \cite{teipel2017robust}, estimating global bank network connectedness \cite{demirer2018estimating}.

Inspired by the fact that the pixel values generated from image restoration problems should be nonnegative, Bai et al. \cite{bai2017modulus} studied a
nonnegative Tikhonov regularization as below:
$$
\min\limits_{x\in\Re^n}~ \frac{1}{2}\|Ax-b\|^2+\frac{\lambda}{2}\|x\|^2+\mathbb{I}_{\Re^{n}_+}(x),
$$
where $\mathbb{I}_{\Re^{n}_+}(x)=0$ if $x\geq0$; $\mathbb{I}_{\Re^{n}_+}(x)=+\infty$ otherwise.

If $\varphi(x):=\mu\|x\|_1+\mathbb{I}_{\Re^{n}_+}(x),\,\,\mu>0$, problem \eqref{RLasso} reduces to the following nonnegative elastic Lasso, which  has been applied to estimate the
microstructure indices from diffusion magnetic resonance data \cite{daducci2015accelerated} and to track the market index \cite{wu2014nonnegative}
\begin{equation}\label{model-NElass}
\min\limits_{x\in\Re^n}~ \frac{1}{2}\|Ax-b\|^2+\frac{\lambda}{2}\|x\|^2+\mu\|x\|_1+\mathbb{I}_{\Re^{n}_+}(x).
\end{equation}
One of the most popular way to find the solution of problem \eqref{RLasso} is to reformulate it into the following regularized least square problem and directly apply some well-established algorithms, such as the block coordinate descent algorithm \cite{friedman2010regularization} and accelerated proximal gradient method \cite{beck2009fast,chen2018generalized}:
\begin{equation}\label{RLasso-R}
\min\limits_{x\in\Re^n}~ f(x):=\frac{1}{2}\|
\widetilde{A}x-\tilde{b}\|^2+\varphi(x),
\end{equation}
where
\begin{equation}\label{def-new-Ab}
\widetilde{A}=\left(\begin{array}{c}
A\\
\sqrt{\lambda}I
\end{array}
\right)\in\Re^{(m+n)\times n},\qquad\tilde{b}=\left(\begin{array}{c}
b\\
0
\end{array}
\right)\in\Re^{m+n}.
\end{equation}
It can be seen from \eqref{def-new-Ab} that the size of newly constructed data matrix $\widetilde{A}$ may increase greatly in the number of rows. This may bring additional issues by using the existing solution methods.

More recently, the semismooth Newton based augmented Lagrangian algorithm (S{\footnotesize SNAL}) has shown excellent numerical performance for
solving various large-scale Lasso-type problems \cite{li2018efficiently,li2016highly,Zhang2018efficient}. {From the relationship between the augmented Lagrangian algorithm and the proximal point algorithm \cite[Section 4]{rockafellar1976augmented}, we can observe that each subproblem of S{\footnotesize SNAL} is essentially a Tikhonov regularization problem (see Remark \ref{remark-1} for the details). Motivated by the observation,
we aim to develop a dual based semismooth Newton method for solving the sparse Tikhonov regularization problem \eqref{RLasso}.}

{It is well known that the local convergence results for convex optimization  are usually not sufficient to guarantee the performance of the semismooth Newton method . An efficient globalization strategy has been proposed in \cite{qi2006quadratically}. This strategy mainly depends on a continuously differentiable convex function
whose gradient is Karush-Kuhn-Tucker (KKT) mapping (see e.g., \cite[Page 625]{han2017linear}) of the original optimization problem.
However, the strategy used in \cite{qi2006quadratically} is no longer applicable to the primal formulation \eqref{RLasso}. The problem is due to the fact that the KKT mapping of problem \eqref{RLasso} (see \eqref{def-F}) can not easily be viewed as a gradient mapping of any real valued function}. On the other hand, though the smoothing Newton method studied in \cite{gao2009calibrating} is globally convergent, additional smoothing function should be introduced. In fact, by taking the advantage of the Tikhonov regularization, we can show that the DSSN method converges globally and at least R-superlinearly.

{The main contributions of this paper can be summarized as follows. Firstly, we propose a dual based semismooth Newton method (DSSN) for solving the sparse Tikhonov regularization problem \eqref{RLasso}. By fully taking advantage of the Tikhonov regularization, the proposed DSSN method can avoid enlarging the scale of data matrix.
Secondly,  we prove that the proposed DSSN method is globally convergent and can achieve at least R-superlinear convergence rate.
Finally, by deeply exploring the sparsity of the second-order information associated with the sparse regularizer, the robustness and effectiveness can be shown in the proposed method. This can be supported by the
numerical results presented in section 4.}

The rest of this paper is organized as follows. In section \ref{sec:preliminary}, we present some preliminary results that will be used for arithmetic design and numerical implementation. In section \ref{sec:DSSN}, we propose the dual based semismooth Newton method for solving the sparse Tikhonov regularization problem, and establish its convergence results. In section \ref{sec:num}, we evaluate the numerical performance of the DSSN on UCI data sets.
Finally, we conclude the paper in section \ref{sec:conclusion}.

{\bf Notation:} Let $\Re^n (\Re^n_+, \Re^n_-$)  be the set of all (non-negative, non-positive) $n$-vectors, $\mathbb{R}^{m\times n}$ be the set of all $m\times n$ real matrices, and $\mathbb{S}^n$ be the set of all $n\times n$ real symmetric matrices.  Let $\mathcal{U}$ be a finite-dimensional real Hilbert space, $h:\mathcal{U}\to\Re\cup\{+\infty\}$, we use $\partial_Bh(u)$ and $\partial h(u)$ to denote the B-subdifferential \cite[Equation (2.12)]{qi1993convergence} and Clarke's generalized Jacobian \cite[Definition 2.6.1]{clarke1990optimization} of function $h$ at $u\in\mathcal{U}$, respectively. Let $C$ be a closed convex set in $\mathcal{U}$, we use $\Pi_{C}(u)$ to denote the Euclidean projection of $u\in\mathcal{U}$ onto $C$. Let ${\rm Diag}(v)$ denote a diagonal matrix whose $i$-th diagonal entry is the $i$-th element of vector $v\in\Re^n$. Let ${\rm sign}(\cdot)$ denote the sign mapping on $\Re^n$, i.e., $[{\rm sign}(v)]_i=1$, if $v_i>0$; $[{\rm sign}(v)]_i=-1$, if $v_i<0$; $[{\rm sign}(v)]_i=0$, if $v_i=0,\, i=1,\ldots,n$. We denote the vector of all ones by $e$.
Let index set $\mathcal{I}\subseteq\{1,\ldots,n\}$, for any matrix $A\in\Re^{m\times n}$, we use $A_{\mathcal{I}}$  to denote $m\times |\mathcal{I}|$ sub-matrix of $A$ obtained by removing all the columns of $A$ not in $\mathcal{I}$.  We use ``$\cdot$'' to denote the Hadamard product between matrices and  ``$\circ$'' to denote function composition.

\section{Preliminaries}\label{sec:preliminary}
Let $\phi: \mathcal{U}\rightarrow{\Re}\cup\{+\infty\}$ be a proper, lower semicontinuous, convex function. Denote by $\Phi_{\phi}(u)$ the Moreau-Yosida regularization \cite{moreau1965proximite,yosida1964functional} of $\phi$,
$$
\begin{array}{c}
\Phi_{\phi}(u):=\min\limits_{u'\in\mathcal{U}}\left\{\phi(u')+\frac{1}{2}\|u'-u\|^2\right\},\,\,\,\forall\,u\in\mathcal{U}.
\end{array}
$$
The proximal mapping associated with $\phi$ is defined by
$$
\begin{array}{c}
{\rm Prox}_{\phi}(u) := \arg\min\limits_{u'\in\mathcal{U}}\left\{\phi(u')+\frac{1}{2}\|u'-u\|^2\right\},\,\,\,\forall\,u\in\mathcal{U}.
\end{array}
$$
From e.g., \cite{hiriart2013convex,lemarechal1997practical}, we know that $\Phi_{\phi}(\cdot)$ is a continuously differentiable convex function with its gradient being given by
$$
\nabla \Phi_{\phi} (u) = u-{\rm Prox}_{\phi}(u).
$$
The following identity \cite[Theorem 31.5]{rockafellar2015convex} will be used in the subsequent analysis,
\begin{equation}\label{prox-id}
{\rm Prox}_{\phi}(u)+{\rm Prox}_{\phi^*}(u)=u,
\end{equation}
where $\phi^*$ is the conjugate function (for its definition, see e.g., \cite[Page 104]{rockafellar2015convex}) of $\phi$.
\begin{lemma}\label{lemma-prof-conj}
For any given $u\in\Re^n$ , the following hold:
\begin{description}
\item[(a)] If $\phi(x)=\mu\|x\|_1$ and $\mu>0$, then $\phi^*(u)=\mathbb{I}_{\mathbb{B}_{\infty,\mu}}(u)$ with $\mathbb{B}_{\infty,\mu}:=\{u|\,\|u\|_{\infty}\leq \mu\}$ and
$$
{{\rm Prox}_{\phi}(u)={\rm sign}(u)\cdot{\rm max}\{|u|-\mu e,0\}.}
$$

\item[(b)] If $\phi(x)=\mathbb{I}_{\Re^{n}_+}(x)$, then  $\phi^*(x)=\mathbb{I}_{\Re^{n}_-}(x)$ and ${\rm Prox}_{\phi}(u)=\Pi_{\Re^n_+}(u) $.
\end{description}
\end{lemma}

\begin{lemma}\label{lem-B-sub}
For any given $v\in\Re^n$, the following hold:
\begin{description}
  \item[(a)] The B-subdifferential of ${\rm Prox}_{\mu\|\cdot\|_1}(\cdot)$ at $v$ is given by
  \begin{equation}\label{def-Bsub-L1}
  \partial_B{\rm Prox}_{\mu\|\cdot\|_1}(v) = \left\{{\rm Diag}(\theta)\Big|\, \theta\in\Re^n,\,\,{\theta_i}\in\left\{\begin{array}{ll}
   \{1\}, & \hbox{\it if $|v_i|>\mu$},\\
   \{0,1\}, & \hbox{\it if $|v_i|=\mu$},\\
 \{ 0\}, & \hbox{\it if $|v_i|<\mu$},
   \end{array}
   \right. \,\,i=1,\ldots,n  \right\}.
  \end{equation}
\item[(b)]The B-subdifferential of $\Pi_{\Re^n_+}(\cdot)$ at $v$ is given by
\begin{equation*}
\partial_B\Pi_{\Re^n_+}(v)=\left\{{\rm Diag}(\theta)|\,\theta\in\Re^n,\,\theta_i\in\partial_B\max\{v_i,0\},\,\,  \,\,i=1,\ldots,n  \right\},
\end{equation*}
where for $i=1,\ldots,n$,
\begin{equation}\label{def-Bsub-pos}
\partial_B\max\{v_i,0\}=\left\{\begin{array}{ll}
	\{1\}, & \hbox{if $v_i>0$},\\
	\{0,1\}, & \hbox{if $v_i=0$},\\
	\{0\}, & \hbox{if $v_i<0$}.
\end{array}
\right.
\end{equation}
\end{description}
\end{lemma}
Since both ${\rm Prox}_{\mu\|\cdot\|_1}(\cdot)$ and $\Pi_{\Re^n_+}(\cdot)$ are piecewise linear functions, we can obtain the following results from \cite[Proposition 7.47]{facchinei2007finite}.
\begin{lemma}\label{lem-prox-semi}
For given $\mu>0$,  both ${\rm Prox}_{\mu\|\cdot\|_1}(\cdot)$ and $\Pi_{\Re^n_+}(\cdot)$ are strongly semismooth.
\end{lemma}

The proof of the following results can be obtained by using the same routine of \cite[Proposition 2.1]{Zhang2018efficient} and \cite[Theorem 2]{li2018efficiently}.
In order to make this paper more readable, we will provide a sketch of the proof of the following lemma.
\begin{lemma}\label{lem-semismooth}
For any given $\mu>0$, let $\phi(x):=\mu\|x\|_1+\Pi_{\Re^{n}_+}(x)$. Then it holds that
\begin{equation}\label{def-prox-dec}
{\rm Prox}_{\phi}(z) = {\rm Prox}_{\mu\|\cdot\|_1}\circ \Pi_{\Re^n_+}(z), \,\,\,\forall z\in\Re^n.
\end{equation}
Furthermore, define a set-valued function $\widehat{\partial}_B{\rm Prox}_{\phi}:\Re^n\rightrightarrows\Re^{n\times n}$ as
\begin{equation}\label{def-subdiff-10}
\widehat{\partial}_B{\rm Prox}_{\phi}(u) = \left\{W=\Theta\Xi:\,\Theta\in \partial_B{\rm Prox}_{\mu\|\cdot\|_1}(v),\,\Xi\in \partial_B\Pi_{\Re^n_+}(u),\,v= \Pi_{\Re^n_+}(u)\right\},
\end{equation}
then the  set-valued function $\widehat{\partial}_B{\rm Prox}_{\phi}$ is a nonempty compact valued upper-semicontinuous multi-function and for any $W\in\widehat{\partial}_B{\rm Prox}_{\phi}(u)$, $W$ is a symmetric positive semidefinite matrix and for  $u'\rightarrow u$, it holds that
\begin{equation}\label{lem-semi-smth}
{\rm Prox}_{\phi}(u')-{\rm Prox}_{\phi}(u)-W(u'-u)=\mathcal{O}(\|u'-u\|^2),\,\,\forall\,W\in \widehat{\partial}_B{\rm Prox}_{\phi}(u').
\end{equation}
\end{lemma}
\begin{proof}
By \cite[Theorem 1]{Yu2013}, in order to obtain equation \eqref{def-prox-dec}, it is sufficient to show that
$$
 \partial \Pi_{\Re^n_+}(z)\subseteq\partial \Pi_{\Re^n_+}({\rm Prox}_{\mu\|\cdot\|_1}(z)),\,\,\,\forall z\in\Re^n,
$$
which can be obtained directly from the definition of Clarke's generalized Jacobian and Lemma \ref{lem-B-sub}.
Furthermore,  it follows from Lemma \ref{lem-prox-semi}, \eqref{def-prox-dec} and \cite[Theorem7.5.17]{facchinei2007finite} that the set-valued function $\widehat{\partial}_B{\rm Prox}_{\phi}$ is a nonempty compact valued upper-semicontinuous multi-function and equation \eqref{lem-semi-smth} holds. Besides, the result that any elements in $\widehat{\partial}_B{\rm Prox}_{\phi}(u)$ is symmetric and positive semidefinite can be obtained from Lemma \ref{lem-B-sub}. The proof is completed.
\end{proof}
\begin{remark}
In the numerical experiment, motivated by the idea in \cite{li2016highly}, for any given $v\in\Re^n$, we choose the following element in $\partial_B{\rm Prox}_{\mu\|\cdot\|_1}(v)$:
\begin{equation}\label{def-prox-L1}
\Theta={\rm Diag}(\theta)\,\,\,\hbox{\it with}\,\,\,
\theta_i=\left\{\begin{array}{ll}
1, & \hbox{if $|v_i|>\mu$},\\[3pt]
0, & \hbox{if $|v_i|\leq \mu$},\,\,\,i=1,\ldots,n,
\end{array}
\right.
\end{equation}
and  choose one element $\Theta\in \partial_B\Pi_{\Re^n_+}(v)$ as follows:
\begin{equation}\label{def-theta}
\Theta={\rm Diag}(\theta)\,\,\,\hbox{\it with}\,\,\,
\theta_i=\left\{\begin{array}{ll}
1, & \hbox{if $v_i>0$},\\[3pt]
0, & \hbox{if $v_i\leq 0$},\,\,\,i=1,\ldots,n.
\end{array}
\right.
\end{equation}
\end{remark}

\section{A Dual based Semismooth Newton Method }\label{sec:DSSN}
By introducing an auxiliary variable, problem \eqref{RLasso} can be reformulated as
\begin{equation}\label{RLasso-P}
\begin{array}{cl}
\underset{x,y}{\min} & \frac{1}{2}\|y\|^2+\frac{\lambda}{2}\|x\|^2+\varphi(x)\\
{\rm s.t.} & Ax-y-b=0.
\end{array}
\end{equation}
The Lagrangian function associated with the above problem is given by
$$
\begin{array}{rl}
\mathcal{L}(x,y,z) & = \frac{1}{2}\|y\|^2+\frac{\lambda}{2}\|x\|^2+\varphi(x)+\langle Ax-y-b,z\rangle\\[6pt]
~ & =\frac{1}{2}\|y-z\|^2-\frac{1}{2}\|z\|^2+\varphi(x)+\frac{\lambda}{2}\|x+\lambda^{-1}A^Tz\|^2
-\frac{1}{2}\|\lambda^{-1}A^Tz\|^2-b^Tz.
\end{array}
$$
Then, we can obtain that
$$
\begin{array}{rl}
\underset{x,y}{\min}~\mathcal{L}(x,y,z) & =\underset{y}{\min}\{\frac{1}{2}\|y-z\|^2-\frac{1}{2}\|z\|^2\}\\
&~~+\underset{x}{\min} \{ \varphi(x)+\frac{\lambda}{2}\|x+\lambda^{-1}A^Tz\|^2-\frac{\lambda}{2}\|\lambda^{-1}A^Tz\|^2\}-b^Tz\\[7pt]
& = -\frac{1}{2}\|z\|^2+\Phi_{\lambda^{-1}\varphi}(-\lambda^{-1}A^Tz)-\frac{\lambda}{2}\|\lambda^{-1}A^Tz\|^2-b^Tz,
\end{array}
$$
where $\Phi_{\lambda^{-1}\varphi}(\cdot)$ is the Moreau-Yosida regularization of $\lambda^{-1}\varphi$.
Consequently, the dual problem (minimization form)  of \eqref{RLasso-P} takes the following form:
\begin{equation}\label{RLasso-D}
\underset{z}{\min}~~h(z):= \frac{1}{2}\|z\|^2-\Phi_{\lambda^{-1}\varphi}(-\lambda^{-1}A^Tz)+\frac{\lambda}{2}\|\lambda^{-1}A^Tz\|^2+b^Tz.
\end{equation}
Then the optimal solution $x^{*}$ of problem \eqref{RLasso} can  be obtained by
\begin{equation}\label{def-opt-x}
x^{*} = {\rm Prox}_{\lambda^{-1}\varphi}(-\lambda^{-1}A^Tz^*),
\end{equation}
where $z^*$ is the unique optimal solution of problem \eqref{RLasso-D}.

Since $h(z)$ is a strongly convex and continuously differentiable function, we know that the unique solution of \eqref{RLasso-D} can be obtained by solving the following linear equation:
\begin{equation}\label{new-sys}
\nabla h(z) =0,
\end{equation}
where
$\nabla h(z)=z+b - A{\rm Prox}_{\lambda^{-1}\varphi}(-\lambda^{-1}A^Tz).$
Furthermore, define

\begin{equation}\label{def-Jac}
\widehat{\partial}_B(\nabla h)(z):=\left\{ I+ \lambda^{-1}AWA^T\in\Re^{m\times m}|\,W\in \widehat{\partial}_B{\rm Prox}_{\lambda^{-1}\varphi}(-\lambda^{-1}A^Tz)\right\},
\end{equation}
where $\widehat{\partial}_B{\rm Prox}_{\lambda^{-1}\varphi}$ denotes the generalized Jacobian of ${\rm Prox}_{\lambda^{-1}\varphi}$.
If $\varphi(\cdot)=\mu\|\cdot\|_1$ or $\varphi(\cdot)=\mathbb{I}_{\Re^{n}_+}(\cdot)$, the generalized Jacobian  $\widehat{\partial}_B{\rm Prox}_{\lambda^{-1}\varphi}$
is exactly the B-subdifferential given by \eqref{def-Bsub-L1} or \eqref{def-Bsub-pos}. If $\varphi(\cdot)=\mu\|\cdot\|_1+\mathbb{I}_{\Re^{n}_+}(\cdot)$,  the generalized Jacobian  $\widehat{\partial}_B{\rm Prox}_{\lambda^{-1}\varphi}$ is defined by \eqref{def-subdiff-10}.

In order to design an implementable semismooth Newton method for solving linear system \eqref{new-sys}, we introduce the following proposition which shows that $\widehat{\partial}_B(\nabla h)$ defined by \eqref{def-Jac} can be used as a surrogate generalized Jacobian of $\nabla h$. Moreover, the following result also plays an important role in establishing  the convergence results of the dual based semismooth Newton
method  (Algorithm \ref{alg:ssn}) for solving problem \eqref{RLasso}.
\begin{proposition}\label{prop-semi}
The set-valued function $\widehat{\partial}_B(\nabla h)$ defined by \eqref{def-Jac} is a nonempty compact valued upper-semicontinuous multi-function and for any $V\in\widehat{\partial}_B(\nabla h)(z)$, $V$ is a symmetric positive definite matrix and for $z'\rightarrow z$, it holds that
\begin{equation}\label{prox-h-semi}
\nabla h(z')-\nabla h(z)-V(z'-z)=\mathcal{O}(\|z'-z\|^2),\,\,\forall\,V\in \widehat{\partial}_B(\nabla h(z')).
\end{equation}
\end{proposition}
\begin{proof}
 From the properties of the Moreau-Yosida regularization, we know that the function $h$ defined in \eqref{RLasso-D} is continuously differentiable with Lipschitz continuous gradient. This together with Lemma \ref{lem-prox-semi}, Lemma \ref{lem-semismooth} and \cite[Proposition 7]{li2018efficiently}
implies equation \eqref{prox-h-semi} holds. The proof is completed.
\end{proof}

Now, we are ready to describe the dual based semismooth Newton
(DSSN) method, which is essentially the
semismooth Newton method for solving \eqref{new-sys}.

{\begin{algorithm}[H]
\caption{\bf {Dual based Semismooth Newton (DSSN) Method  for Solving \eqref{RLasso} }}
\label{alg:ssn}
\vspace{1mm}
Given $\varrho \in (0,1/2)$, $\bar{\eta} \in (0,1)$, $\zeta \in (0,1]$, and $\beta \in (0,1)$. Choose $z^0\in\Re^m$. For $j=0,1,\dots$, perform the following steps in each iteration.
\begin{description}
\item[Step 1.] (Newton direction) Choose one specific matrix  $V_j\in\widehat{\partial}_B(\nabla h)(z^j)$. Specifically, we choose
an element $W_j\in {\partial}_B{\rm Prox}_{\lambda^{-1}\varphi}(-\lambda^{-1}A^Tz^j)$ and then
$ V_j= I+ \lambda^{-1}{A} W_j {A}^T\in\widehat{\partial}_B(\nabla h)(z^j) $. Solve the following linear system
\begin{equation}\label{newton-sys}
V_j d = -\nabla h(z^j)
\end{equation}
exactly or by the conjugate gradient (CG) algorithm to find $d^j$ such that
$\| V_j d^j+\nabla h(z^j)\| \leq \min(\bar{\eta},\|\nabla h(z^j)\|^{1+\zeta})$.
\item[Step 2.] (Line search) Set $ \alpha_j = \beta^{l_j}$, where $l_j$ is the smallest nonnegative integer $l$ for which
\begin{equation}\label{ssn-LinSear}
h(z^j + \beta^l d^j) \leq h(z^j) + \varrho\beta^l \langle \nabla h(z^j),d^j \rangle.
\end{equation}
\item[Step 3.] Set $ z^{j+1} = z^j + \alpha_j d^j $  and $x^{j+1} = {\rm Prox}_{\lambda^{-1}\varphi}(-\lambda^{-1}A^Tz^{j+1})$.		
\end{description}
\end{algorithm}}	
The following theorem provides the global convergence of the sequence $\{(z^j,x^j)\}$, the local superlinear convergence rate of $\{z^j\}$, and the local R-superlinear convergence rate of $\{x^j\}$.
\begin{theorem}\label{th-conv}
Let $\{(z^j,x^j)\}$ be the sequence generated by Algorithm \ref{alg:ssn}. Then the following hold:
\begin{itemize}[leftmargin=0.7cm]
\item[\rm (a)] The sequence $\{z^j\}$ is well-defined and  converges globally to the unique solution ${z}^*$ of \eqref{new-sys}. Moreover, the convergence rate is at least superlinear:
\begin{equation}\label{th-rate-z}
\|z^{j+1} - {z}^*\| = \mathcal{O}(\|z^j - {z}^*\|^{1+\zeta}),
\end{equation}
where $\zeta\in (0,1]$ is the parameter given in Algorithm \ref{alg:ssn}.
\item[\rm(b)] The sequence $\{x^j\}$ converges globally to the unique solution $x^*$ of problem \eqref{RLasso} with at least R-superlinear convergence rate.
\end{itemize}
\end{theorem}
{\begin{proof}Since the function $h$ is strongly convex, we can obtain result $(a)$ from Proposition \ref{prop-semi} by following the same routine in the proof of \cite[Theorem 3]{li2018efficiently}.
Now, we come to result $(b)$. It holds from the definition of $x^j$ and  \eqref{def-opt-x} that
$$
\begin{array}{ll}
\|x^{j+1}-x^*\|&=\|{\rm Prox}_{\lambda^{-1}\varphi}(-\lambda^{-1}A^Tz^{j+1})-{\rm Prox}_{\lambda^{-1}\varphi}(-\lambda^{-1}A^Tz^{*})\|\\[2mm]
~& \leq \lambda^{-1}\sigma_{\max}(A)\|z^{j+1}-z^*\|,
\end{array}
$$
where $\sigma_{\max}(A)$ is the largest singular value of $A$.  This together with Result $(a)$ and \eqref{th-rate-z} implies that the sequence $\{x^j\}$ converges globally to the unique solution $x^*$ and there exists a positive scaler $\kappa$ such that
$$
\|x^{j+1}-x^*\|\leq \kappa\|(z^{j}-z^*)^{1+\zeta}\|.
$$
Consequently, for sufficient large $j$, there exists $\kappa'\in(0,1)$ such that
$$
\|x^{j+1}-x^*\|\leq \kappa'\|z^{j}-z^*\|.
$$
This completes the proof.
\end{proof}

\begin{remark}\label{remark-1} We mentioned in the introduction that solving the sparse Tikhonov regularization problem can be viewed as solving a subproblem of the  S{\footnotesize SNAL} \cite{li2018efficiently,li2016highly,Zhang2018efficient}. Here, we give some comments on this point by taking the elastic net Lasso \eqref{model-Elass} for example.
Observe that, for the elastic net Lasso, problem \eqref{RLasso-D} can be viewed as a subproblem corresponding to each iteration of S{\footnotesize SNAL} which was proposed in \cite{li2016highly}. Specifically, from the discussions presented in \cite[Section 4]{rockafellar1976augmented}, we know that the outer iteration scheme of  S{\footnotesize SNAL} is equivalent to the proximal point algorithm (PPA) described as follows: given positive scalars $\sigma_k\uparrow\sigma_{\infty}\leq \infty$, the $k$-th iteration of PPA for Lasso problem is given by
$$
x^{k+1}=\arg\min\limits_{x\in\Re^n}\frac{1}{2}\|Ax-b\|^2+\mu\|x\|_1+\frac{1}{2\sigma_k}\|x-x^k\|^2.
$$
Therefore, the optimization problem corresponding to the $k$-th iteration of PPA  can reduce to the elastic net model when $x^k=0$ and $\sigma_k=\lambda^{-1}$.
\end{remark}

\subsection{Techniques for Linear System}\label{sec-newton-sys}
The linear system \eqref{newton-sys} in Algorithm \ref{alg:ssn} is the most time-consuming part. We take the nonnegative elastic net \eqref{model-NElass} for example, i.e.,
$$
\varphi(x)=\mu\|x\|_1+\mathbb{I}_{\Re^n_+}(x).
$$
In this case, for any given $\tilde{z}\in\Re^n$, the element $\Theta\in\partial_B{\rm Prox}_{\varphi}(\tilde{z})$ is chosen as
\begin{equation}\label{jacob}
\Theta = \Theta_1\Theta_2,
\end{equation}
where $\Theta_1$ and $\Theta_2$ are given by \eqref{def-prox-L1} and \eqref{def-theta}, respectively. Since both $\Theta_1$ and $\Theta_2$ are diagonal matrices, the matrix $\Theta$ is a diagonal matrix, i.e. $\Theta:={\rm Diag}(\theta)$. Then, we define
\begin{equation}\label{def-indx-P}
\mathcal{I}(\tilde{z}):=\{i:\,\theta_{i}=1\},\,\,\,\mathcal{I}_0(\tilde{z}):=\{i:\,\theta_{i}=0\},
\end{equation}
where $\tilde{z}:= -\lambda^{-1}A^Tz$.
Therefore, the linear system \eqref{newton-sys} can be simplified as follow:
\begin{equation}\label{new-sys-simple}
\begin{array}{l}
\left(I_m+\lambda^{-1}A_{\mathcal{I}(\tilde{z})}A^T_{\mathcal{I}(\tilde{z})}\right)d= -\nabla h(z),
\end{array}
\end{equation}
where $A_{\mathcal{I}(\tilde{z})}\in\Re^{m\times |{\mathcal{I}(\tilde{z})}|}$.

Since the coefficient matrix of equation \eqref{new-sys-simple} is symmetric and positive definite, there are many solution methods for finding the solution of the linear system, see e.g., \cite{golub2012matrix,lam2018fast}.
Next, we recall some techniques that can usually be used to increase the computation efficiency.
\begin{itemize}[leftmargin=5mm]
\item[-] If $m<<|{\mathcal{I}(\tilde{z})}|$ and $m$ is moderate, the linear equations \eqref{new-sys-simple} can be solved by using Cholesky factorization \cite[Theorem 4.2.7]{golub2012matrix} of the coefficient matrix.

\item[-] If $|{\mathcal{I}(\tilde{z})}|<<m$ and $|{\mathcal{I}(\tilde{z})}|$ is moderate, the Sherman-Morrison-Woodbury formula \cite[Page 50]{golub2012matrix} can be used to reduce the computational efforts. Specifically,
    $$
    \left(I_m+\lambda^{-1}A_{\mathcal{I}(\tilde{z})}A^T_{\mathcal{I}(\tilde{z})}\right)^{-1}=I_m-\lambda^{-1}A_{\mathcal{I}(\tilde{z})}
    (I+\lambda^{-1}A^T_{\mathcal{I}(\tilde{z})}A_{\mathcal{I}(\tilde{z})})^{-1}A^T_{\mathcal{I}(\tilde{z})}.
    $$
Therefore, the computation cost of solving linear system \eqref{new-sys-simple} can be reduced by applying Cholesky factorization of the matrix $I+\lambda^{-1}A^T_{\mathcal{I}(\tilde{z})}A_{\mathcal{I}(\tilde{z})}\in\mathbb{S}^{|{\mathcal{I}(\tilde{z})}|}$.

\item[-] If both $|{\mathcal{I}(\tilde{z})}|$ and $m$ are large, the practical conjugated gradient algorithm \cite[Algorithm 10.2.1]{golub2012matrix} can be applied to solve the linear system \eqref{new-sys-simple} approximately.

\end{itemize}

\section{Numerical Experiments}\label{sec:num}
In this section, we compare the numerical performance of the dual based semismooth Newton method (DSSN) and the primal based semismooth Newton method (PSSN) on the instances from UCI data repository \cite{lichman2013uci} . In the numerical experiments, the data sets {\it mgp7}, {\it pyrim5}, and {\it bodyfat7} are the expended data sets by using the polynomial basis function. For the details of data processing, we refer to \cite{huang2010predicting,li2016highly}.

The first order optimal condition of problem \eqref{RLasso} can be written as
\begin{equation}\label{def-F}
\Psi(x):=x-{\rm Prox}_{\varphi}\left(x-\lambda x- A^T(Ax-b)\right)=0.
\end{equation}
 Based on the above condition, we measure the accuracy of the approximate solution of problem \eqref{RLasso} by using the following residual:
$$
\eta = \|x-{\rm Prox}_{\varphi}\left(x-\lambda x- A^T(Ax-b)\right)\|.
$$
For a given tolerance $\varepsilon=10^{-6}$, all the tested algorithms  will be stopped when $\eta\leq\varepsilon$ or the maximum number 200 of iterations is  reached. Inspired by the parameter used in \cite{zou2005regularization}, we choose $\lambda\in\{100,10,1,0.1,0.01\}$. For testing purpose, the parameter $\mu$ is chosen as
$$
\mu = \mu_c\|A^Tb\|_{\infty}.
$$
All our numerical results are obtained by running M{\footnotesize{ATLAB}} R2018b on a desktop (4-core, i5-7300 CPU @2.60GHz, 8.00GB of RAM).

\subsection{Primal based Semismooth Newton Method}
Note that $\Psi(x)$ defined by \eqref{def-F} is a locally Lipschitz continuous function and one can employ the following semismooth Newton method (SSN):
$$
x^{k+1} = x^k-V^{-1}_k \Phi(x^k),
$$
where $V_k\in\widehat{\partial}_B \Psi(x^k)$ with the mapping $\widehat{\partial}_B \Psi:\Re^n\rightrightarrows\mathbb{S}^n$ being defined by
$$
\widehat{\partial}_B \Psi(\bar{x}) :=\left\{I-\Theta\big((1-\lambda)I-A^TA\big)\in\mathbb{S}^n:\Theta\in\widehat{\partial}_B{\rm Prox}_{\varphi}(\bar{x})  \right\},\,\,\,\forall\,\bar{x}\in\Re^n.
$$
Inspired by the smoothing Newton method (see e.g., \cite{facchinei2007finite,sun2001solving}) and the solution method used in \cite{zhao2018semismooth}, we define the merit function
\begin{equation}\label{def-merit}
r(x) = \|\Psi(x)\|^2.
\end{equation}
Therefore, the primal based semismooth Newton method for solving problem \eqref{RLasso} can be described by Algorithm \ref{alg:ssn-P}.
{\begin{algorithm}[H]
\caption{\bf Primal based Semismooth Newton (PSSN) Method for Solving \eqref{RLasso} }
\label{alg:ssn-P}
\vspace{1mm}
Given $\varrho\in (0,1/2)$, and $\beta \in (0,1)$. Choose $y^0\in\Re^m$. For $k=0,1,\dots$, perform the following steps in each iteration.
\begin{description}
\item[Step 1.] (Newton direction) Choose $\Theta\in\partial_B{\rm Prox}_{\varphi}(y)$ with $y=x^j-\lambda x^j- A^T(Ax^j-b)$.  Compute $d^j$ by solving
    \begin{equation}\label{LS-P}
    \big(I-\Theta\big((1-\lambda)I-A^TA\big)\big)d=-\Psi(x^k).
    \end{equation}
\item[Step 2.] (Line search) Set $ \alpha_j = \beta^{l_j}$, where $l_j$ is the smallest nonnegative integer $l$ for which
$$
r(x^j + \beta^l d^j) \leq r(x^j) + \varrho\beta^l \langle \nabla r(x^j),d^j \rangle.
$$
\item[Step 3.] Set $ x^{j+1} = x^j + \alpha_j d^j $.		
\end{description}
\end{algorithm}}	
Similar to Algorithm \ref{alg:ssn}, the linear system \eqref{LS-P} in Algorithm \ref{alg:ssn-P} is the most time-consuming part. We also take the nonnegative elastic net for example. In this case, for any given $y\in\Re^n$, the element $\Theta\in\partial_B{\rm Prox}_{\varphi}({y})$ is chosen as \eqref{jacob}. Define
\begin{equation}\label{def-indx-P}
\mathcal{I}({y}):=\{i:\,\theta_{i}=1\},\,\,\,\mathcal{I}_0({y}):=\{i:\,\theta_{i}=0\},
\end{equation}
where
$$y=x-\lambda x- A^T(Ax-b).$$
Therefore, the linear system \eqref{LS-P} can be simplified as follow:
\begin{align}
&d_{\mathcal{I}_0({y})} = -\Psi_{\mathcal{I}_0({y})}(x^k),\nonumber\\[6pt]
&\big(\lambda {I}_{\mathcal{I}(y)}+A^T_{\mathcal{I}({y})}A_{\mathcal{I}({y})}\big)d_{\mathcal{I}({y})} =  -\Psi_{\mathcal{I}({y})}(x^k)-A^T_{\mathcal{I}({y})}A_{\mathcal{I}_0({y})}d_{\mathcal{I}_0({y})}. \label{pri-new-sys}
\end{align}
Note that the linear system \eqref{pri-new-sys} can be solved by using the similar techniques presented in  subsection \ref{sec-newton-sys}.
\begin{remark}
It follows from \cite[Theorem 3.2]{qi1993nonsmooth} that the PSSN method without line search converges locally quadratically to the unique solution of problem \eqref{RLasso}. The numerical experience shows that the line search step (Step 2 in Algorithm \ref{alg:ssn-P}) can increase the convergence speed. However, since the smoothness of the merit function $r(\cdot)$ defined by \eqref{def-merit} is unclear, it is still a challenge problem that whether the PSSN method can converge globally \cite[Theorem 8.3.15]{Facchinei2007}.
\end{remark}

\subsection{Elastic Net Lasso}
Motivated by \cite{tibshirani2012strong}, the elastic net Lasso problem \eqref{model-Elass} can be reformulated into the following form which is exactly the Lasso problem \cite{tibshirani1996regression}:
\begin{equation}\label{Lasso-R}
\min\limits_{x\in\Re^n}~ f(x):=\frac{1}{2}\|
\widetilde{A}x-\tilde{b}\|^2+\mu\|x\|_1,
\end{equation}
where $\widetilde{A}$ and $\tilde{b}$ are given by \eqref{def-new-Ab}. Therefore, the semismooth Newton based augmented Lagrangian method (S{\footnotesize SNAL}) proposed by \cite{li2016highly} can be used to solve
problem \eqref{RLasso-R}. The high efficiency of S{\footnotesize SNAL}\footnote{Matlab code is available at: \url{http://www.math.nus.edu.sg/~ mattohkc/SuiteLasso.html}} for solving the Lasso problem has been convincingly demonstrated by various data sets by \cite{li2016highly}. Therefore, we take the performance of S{\footnotesize SNAL} for solving problem \eqref{Lasso-R} as a benchmark.

In this section, we present the performance of the DSSN, PSSN, and S{\footnotesize SNAL} for solving the elastic net problem \eqref{model-Elass} on the UCI datasets that mentioned at the beginning of this section.
In order to test the robustness of algorithms with respect to the parameter $\lambda$, for each $\mu_c$,  we choose a sequence of parameters $\{\lambda_k\}$ which run from 1 to 0.1 and solve a sequence of optimization problems by using PSSN, DSSN and S{\footnotesize SNAL}. It should be mentioned that, for each $\mu_c$, the optimal solution of  optimization problem with $\lambda_k$ is taken as the initial point of the  problem with $\lambda_{k+1}$. These results are reported in Table \ref{Table:ELasso1}. In Table \ref{Table:ELasso2}, we select a grid of values of $\lambda\in\{100,10,1,0.1,0.01\}$ and $\mu=k\mu_c(2)$ (here, $\mu_c(2)$ is the second $\mu_c$ that used in Table \ref{Table:ELasso1} for each data set) with $k=100:-1:1$. In this table, for each data set, we only need to choose the initial point once and take optimal solution of the $l$-th problem (say $x^{l*}$) as the initial point of the next optimization problem.

From Tables \ref{Table:ELasso1} and \ref{Table:ELasso2}, we can see that the convergence speed of DSSN of all the test examples are superior to that of PSSN. In order to obtain the index sets \eqref{def-indx-P}, which can be used to reduce the computational effort, the $u=Ax$ and $A^Tu$ have to be computed. The computation of $A^TAu$ is at the cost of ${\rm O}(mn^2)$. However, for the DSSN, the main cost is in computing $A^Tz$ at ${\rm O}(mn)$. For most cases, especially for the higher dimentional problems, DSSN is also outperform the S{\footnotesize SNAL} for solving problem \eqref{Lasso-R}. The most possible reason is the formulation \eqref{Lasso-R} overlooks the advantage of the Tikhonov regularization in designing numerical algorithm and destroys the low-sample size and high-dimensional structure of data matrix $A$.

Furthermore, from the convergence result Theorem \ref{th-conv}, we know that a well-chosen initial point can greatly speed up the convergence.  Based on the results presented in Table \ref{Table:ELasso2}, two instances are chosen to illustrate that the DSSN  has the advantage over PSSN. Figure \ref{Fig:ELasso} shows that the solution gap between the two neighbouring dual optimization problems and the iteration numbers have a similar performance. Compare to the performance of PSSN, both the solution gap and the iteration numbers
are relatively small and stable, especially from the $80$-th to the $100$-th optimization problems.

{\small\begin{table}[H]
\centering
\caption{The performances of DSSN(D), PSSN(P) and {S\footnotesize{SNAL}}(L)  on selected  instances from UCI data repository. $\lambda=1:-0.1:0.1$
}\label{Table:ELasso1}\centering
\vskip 1.5mm
{\small \begin{tabular}{ccccc}
\hline
problem name     & $\mu_c$ &  nnz         &      time     & \multicolumn{1}{c}{error}\\
$(m,n) $       &   ~   & ({\rm max},{\rm min})             &    $D$ $|$ $P$ $|$ $L$     & $D$ $|$ $P$ $|$ $L$\\
\hline
\multirow{3}{*}{\parbox[c]{2.3cm}{Leukaemia\\(38,7129)}} & 1e-02  &  (104, 39) &  0.21 $|$ 0.41 $|$  5.42   &  2.23e-13 $|$ 1.21e-13 $|$ 5.31e-07 \\
 & 1e-03  &  (422, 105) &  0.16 $|$ 0.38 $|$  5.34   &  1.13e-13 $|$ 1.69e-13 $|$ 4.30e-07 \\
 & 1e-04  &  (1924, 423) &  0.24 $|$ 0.44 $|$  5.10   &  3.01e-13 $|$ 2.26e-08 $|$ 4.36e-07 \\\hline
 \multirow{3}{*}{\parbox[c]{2.3cm}{duke\\
 (44,7129)}} & 1e-02  &  (100, 48) &  0.11 $|$ 0.36 $|$  5.83   &  3.29e-13 $|$ 6.62e-14 $|$ 5.82e-07 \\
 & 1e-03  &  (387, 102) &  0.14 $|$ 0.42 $|$  5.56   &  5.76e-13 $|$ 1.15e-13 $|$ 5.09e-07 \\
 & 1e-04  &  (1872, 388) &  0.28 $|$ 0.51 $|$  6.04   &  4.26e-12 $|$ 3.00e-08 $|$ 5.89e-07 \\\hline
 \multirow{3}{*}{\parbox[c]{2.3cm}{colon-cancer\\
 (62,2000)}} & 1e-02  &  (73, 63) &  0.06 $|$ 0.20 $|$  2.57   &  2.74e-12 $|$ 8.46e-14 $|$ 4.65e-07 \\
 & 1e-03  &  (317, 82) &  0.11 $|$ 0.19 $|$  3.59   &  6.55e-12 $|$ 1.34e-13 $|$ 4.86e-07 \\
 & 1e-04  &  (1086, 318) &  0.24 $|$ 0.26 $|$  3.19   &  3.81e-11 $|$ 9.15e-13 $|$ 4.56e-07 \\\hline
 \multirow{3}{*}{\parbox[c]{2.3cm}{mpg7\\
 (392,3432)}} & 1e-02  &  (15, 14) &  0.16 $|$ 0.75 $|$  1.95   &  1.00e-08 $|$ 1.23e-12 $|$ 2.43e-07 \\
  & 1e-03  &  (78, 51) &  0.17 $|$ 0.57 $|$  7.26   &  4.86e-10 $|$ 3.61e-12 $|$ 6.12e-07 \\
   & 1e-04  &  (516, 214) &  0.54 $|$ 0.96 $|$  11.68   &  1.29e-10 $|$ 7.82e-12 $|$ 5.85e-07 \\\hline
 \multirow{3}{*}{\parbox[c]{2.3cm}{pyrim5\\
 (74,201376)}} & 1e-03  &  (422, 166) &  3.41 $|$ 18.24 $|$  152.50   &  2.19e-12 $|$ 1.95e-08 $|$ 5.16e-07 \\
 & 1e-04  &  (2677, 610) &  3.62 $|$ 10.43 $|$  270.93   &  9.45e-12 $|$ 1.09e-07 $|$ 6.58e-07 \\
 & 1e-05  &  (16306, 2803) &  5.14 $|$ 9.35 $|$  424.01   &  6.28e-11 $|$ 8.49e-08 $|$ 5.31e-07 \\\hline
 \multirow{3}{*}{\parbox[c]{2.3cm}{bodyfat7\\
 (252,116280)}} & 1e-03  &  (12, 3) &  2.39 $|$ 9.02 $|$  28.45   &  3.10e-12 $|$ 8.99e-14 $|$ 4.85e-07 \\
 & 1e-04  &  (196, 16) &  2.89 $|$ 11.10 $|$  75.92   &  1.30e-13 $|$ 9.75e-14 $|$ 4.96e-07 \\
 & 1e-05  &  (2734, 206) &  3.19 $|$ 12.90 $|$  129.51   &  2.95e-13 $|$ 3.57e-08 $|$ 5.80e-07 \\\hline
 \multirow{3}{*}{\parbox[c]{2.3cm}{triazines4\\
 (186,635376)}} & 1e-01  &  (321, 321) &  76.34 $|$ 478.34 $|$  25.39   &  1.62e-09 $|$ 6.26e-10 $|$ 1.01e-09 \\
  & 1e-02  &  (416, 387) &  67.53 $|$ 314.15 $|$  748.72   &  3.85e-09 $|$ 2.89e-09 $|$ 5.19e-07 \\
  & 1e-03  &  (1535, 853) &  61.10 $|$ 340.99 $|$  2630.55   &  2.85e-09 $|$ 2.61e-09 $|$ 6.54e-07 \\\hline 	
\end{tabular}}	
\end{table}	}

{\small\begin{table}[H]
\centering
\caption{The performances of DSSN(D), PSSN(P) and {S\footnotesize{SNAL}}(L) on selected instances from UCI data repository.
}\label{Table:ELasso2}\centering
\vskip 1.5mm
\begin{tabular}{ccccc}
\hline
problem name     & $\lambda$ &  nnz         &      time     & \multicolumn{1}{c}{error}\\
$(m,n) $        &   ~   & ({\rm max},{\rm min})             &   $D$ $|$ $P$ $|$ $L$     & $D$ $|$ $P$ $|$ $L$\\
\hline
\multirow{3}{*}{\parbox[c]{2.3cm}{Leukaemia\\(38,7129)}} & 1e+02  &  (6685, 1813) &  3.49 $|$ 4.82 $|$  0.38   &  2.21e-09 $|$ 1.29e-08 $|$ 6.44e-07 \\
 & 1e+01  &  (5336, 413) &  1.64 $|$ 3.85 $|$  0.72   &  2.56e-13 $|$ 1.10e-08 $|$ 4.45e-07 \\
 & 1e+00  &  (1924, 104) &  0.83 $|$ 3.14 $|$  1.13   &  1.28e-13 $|$ 1.08e-09 $|$ 5.76e-07 \\
 & 1e-01  &  (423, 38) &    0.80 $|$ 3.59 $|$  1.66   &  2.65e-13 $|$ 8.62e-10 $|$ 6.54e-07 \\
 & 1e-02  &  (105, 34) &    1.76 $|$ 9.22 $|$  2.49   &  3.30e-12 $|$ 1.34e-13 $|$ 6.47e-07 \\
 \hline
 \multirow{3}{*}{\parbox[c]{2.3cm}{duke\\(44,7129)}} & 1e+02  &  (6624, 1791) &  4.55 $|$ 5.70 $|$  2.53   &  7.71e-13 $|$ 2.61e-08 $|$ 6.67e-07 \\ & 1e+01  &  (5215, 381) &  1.96 $|$ 3.95 $|$  2.61   &  3.65e-13 $|$ 1.92e-08 $|$ 7.47e-07 \\ & 1e+00  &  (1872, 100) &  1.01 $|$ 3.61 $|$  2.68   &  2.90e-13 $|$ 8.08e-09 $|$ 6.99e-07 \\ & 1e-01  &  (388, 47) &  0.86 $|$ 4.16 $|$  2.77   &  7.86e-13 $|$ 2.11e-13 $|$ 6.32e-07 \\ & 1e-02  &  (102, 44) &  1.66 $|$ 8.93 $|$  2.87   &  5.96e-12 $|$ 9.21e-13 $|$ 5.96e-07 \\\hline \multirow{3}{*}{\parbox[c]{2.3cm}{colon-cancer\\(62,2000)}} & 1e+02  &  (1918, 988) &  2.83 $|$ 3.43 $|$  2.67   &  5.54e-13 $|$ 5.16e-10 $|$ 4.93e-07 \\ & 1e+01  &  (1792, 288) &  1.59 $|$ 1.95 $|$  2.47   &  1.35e-12 $|$ 2.76e-08 $|$ 5.30e-07 \\ & 1e+00  &  (1086, 73) &  0.83 $|$ 1.39 $|$  2.24   &  3.04e-12 $|$ 3.90e-09 $|$ 5.75e-07 \\ & 1e-01  &  (318, 63) &  0.72 $|$ 1.42 $|$  2.01   &  1.08e-11 $|$ 4.11e-09 $|$ 5.02e-07 \\ & 1e-02  &  (84, 56) &  1.41 $|$ 2.48 $|$  1.86   &  8.68e-11 $|$ 6.75e-13 $|$ 5.15e-07 \\\hline \multirow{3}{*}{\parbox[c]{2.3cm}{mpg7\\(392,3432)}} & 1e+02  &  (2557, 76) &  1.90 $|$ 4.56 $|$  2.11   &  7.32e-12 $|$ 5.56e-12 $|$ 5.42e-07 \\ & 1e+01  &  (1234, 20) &  1.06 $|$ 4.04 $|$  3.19   &  1.16e-11 $|$ 4.02e-12 $|$ 3.99e-07 \\ & 1e+00  &  (516, 14) &  0.93 $|$ 4.08 $|$  4.40   &  3.38e-10 $|$ 3.08e-12 $|$ 2.50e-07 \\ & 1e-01  &  (214, 14) &  1.04 $|$ 4.76 $|$  6.28   &  1.92e-08 $|$ 2.77e-12 $|$ 3.97e-07 \\ & 1e-02  &  (142, 13) &  1.61 $|$ 6.87 $|$  9.11   &  6.69e-09 $|$ 2.79e-12 $|$ 4.87e-07 \\\hline \multirow{3}{*}{\parbox[c]{2.3cm}{pyrim5\\(74,201376)}} & 1e+02  &  (143811, 12077) &  69.50 $|$ 121.99 $|$  44.24   &  8.72e-09 $|$ 1.07e-07 $|$ 5.67e-07 \\ & 1e+01  &  (74869, 2021) &  31.26 $|$ 70.58 $|$  109.45   &  3.88e-12 $|$ 7.21e-08 $|$ 6.00e-07 \\ & 1e+00  &  (16306, 422) &  18.93 $|$ 77.04 $|$  159.03   &  3.81e-12 $|$ 5.30e-08 $|$ 5.92e-07 \\ & 1e-01  &  (2803, 166) &  18.26 $|$ 106.93 $|$  201.44   &  7.33e-12 $|$ 1.83e-08 $|$ 5.52e-07 \\ & 1e-02  &  (558, 86) &  25.07 $|$ 402.05 $|$  274.29   &  2.59e-11 $|$ 3.82e-08 $|$ 4.05e-07 \\\hline \multirow{3}{*}{\parbox[c]{2.3cm}{bodyfat7\\(252,116280)}} & 1e+02  &  (94185, 2345) &  39.44 $|$ 86.42 $|$  249.40   &  1.63e-09 $|$ 1.87e-07 $|$ 3.03e-07 \\ & 1e+01  &  (28754, 142) &  18.83 $|$ 70.91 $|$  198.32   &  6.25e-09 $|$ 4.39e-08 $|$ 4.02e-07 \\ & 1e+00  &  (2734, 11) &  12.53 $|$ 58.21 $|$  163.72   &  2.96e-13 $|$ 7.50e-14 $|$ 4.65e-07 \\ & 1e-01  &  (206, 3) &  10.96 $|$ 56.01 $|$  133.28   &  5.39e-12 $|$ 6.01e-14 $|$ 4.15e-07 \\ & 1e-02  &  (35, 2) &  11.79 $|$ 73.53 $|$  75.97   &  2.44e-10 $|$ 3.52e-14 $|$ 4.44e-07 \\\hline \multirow{3}{*}{\parbox[c]{2.3cm}{triazines4\\(186,635376)}} & 1e+02  &  (17281, 345) &  145.23 $|$ 826.04 $|$  410.37   &  5.20e-11 $|$ 4.19e-09 $|$ 4.61e-07 \\ & 1e+01  &  (4050, 321) &  96.41 $|$ 623.63 $|$  850.76   &  3.75e-10 $|$ 7.84e-10 $|$ 3.76e-07 \\ & 1e+00  &  (1535, 321) &  119.17 $|$ 801.84 $|$  1535.02   &  3.19e-09 $|$ 1.37e-08 $|$ 3.78e-07 \\ & 1e-01  &  (853, 321) &  190.86 $|$ 1703.73 $|$  2467.18   &  3.34e-08 $|$ 1.19e-07 $|$ 3.95e-07 \\ & 1e-02  &  (653, 321) &  573.12 $|$ 6200.19 $|$  3722.94   &  2.73e-07 $|$ 6.05e-01 $|$ 4.38e-07 \\\hline	
\end{tabular}	
\end{table}	}

\begin{figure}[H]
\includegraphics[scale=0.5]{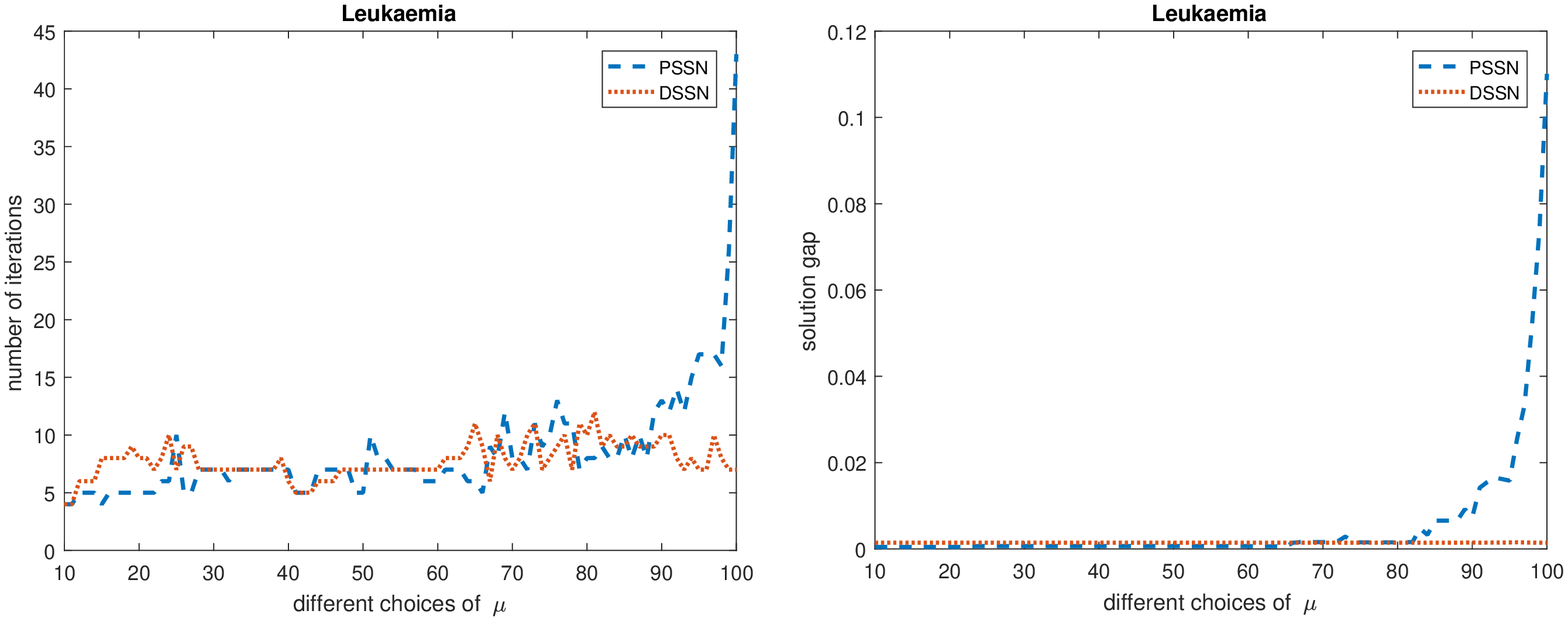}
\includegraphics[scale=0.5]{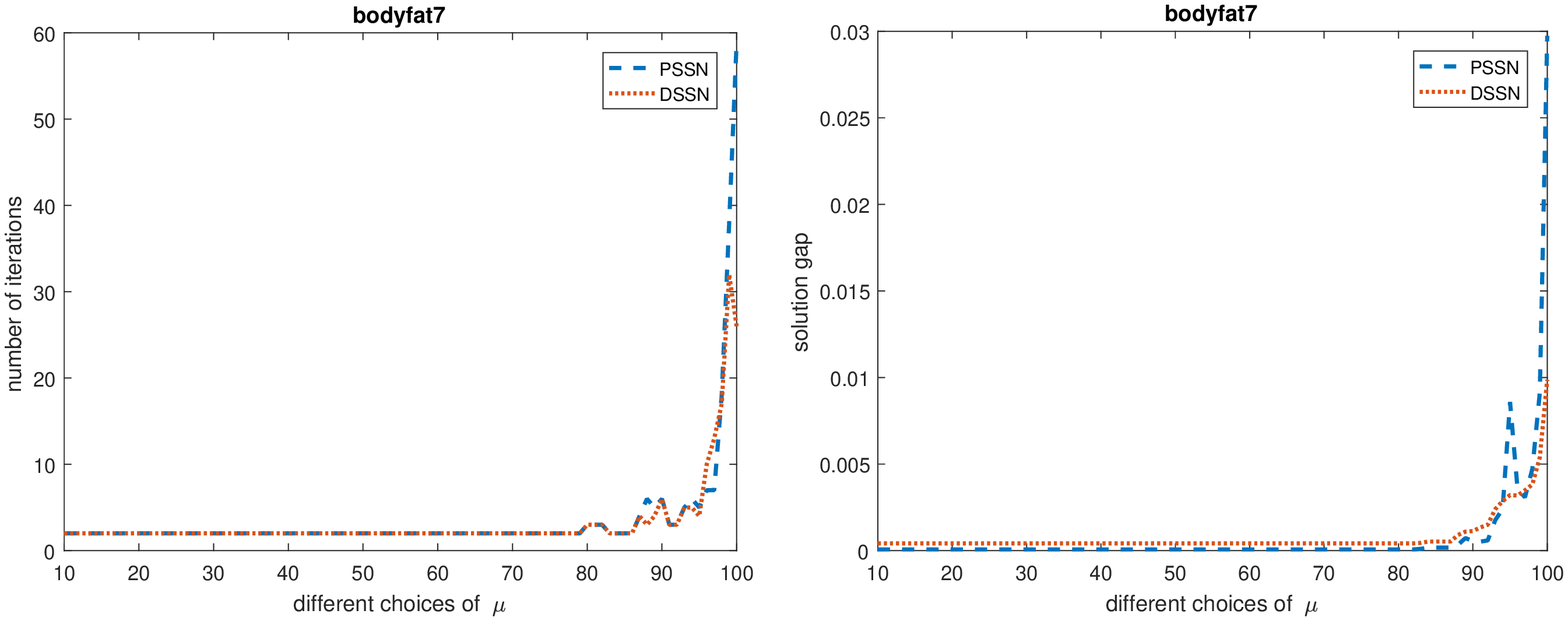}
\caption{Number of iterations and solution gap ( $\|x^{l*}-x^{(l-1)*}\|,\,\,l\geq 10$), on the {\it Leukaemia} and {\it bodyfat7} data sets. $\lambda=0.01$ .}\label{Fig:ELasso}
\end{figure}

\subsection{Nonnegative Elastic Net Lasso}
In this section, we present the performance of the DSSN and PSSN  for solving the nonegative elastic net problem \eqref{model-NElass} on the UCI datasets.\footnote{Due to the nonnegative constraints, the package {\it SuiteLasso} is can not be applied directly.}
These numerical results are presented in Table \ref{Table:NEN1} and Table \ref{Table:NEN2}. Both tables show that the convergence speed of DSSN of all the test examples are superior to that of PSSN. Based on the results presented in Table \ref{Table:NEN2}, two instances {\it Leukaemia} and $\it bodyfat7$ are chosen to illustrate that the DSSN  has the advantage over PSSN and the illustration is presented by Figure \ref{Fig:NELasso}. We can observe that the performance of DSSN for solving the nonnegative elastic net Lasso is similar to that of DSSN for solving the elastic net Lasso.
{\small\begin{table}[H]
\centering
\caption{The performances of DSSN(D) and PSSN(P) on selected instances from UCI data repository. $\lambda=1:-0.01:0.1$
}\label{Table:NEN1}\centering
\vskip 1.5mm
{\small \begin{tabular}{ccccc}
\hline
problem name     & $\mu$ &  nnz         &      time     & \multicolumn{1}{c}{error}\\
$(m,n) $;        &   ~   & ({\rm max},{\rm min})             &   $P$ $|$ $D$     & $P$ $|$ $D$\\
\hline
\multirow{3}{*}{\parbox[c]{2.3cm}{Leukaemia\\(38,7129)}} & 1e-02  &  (94, 44) &  1.03 $|$ 0.15   &  2.69e-14 $|$ 1.39e-13 \\ & 1e-03  &  (369, 97) &  1.47 $|$ 0.33   &  1.72e-08 $|$ 1.24e-13 \\ & 1e-04  &  (1280, 370) &  1.58 $|$ 0.64   &  3.76e-08 $|$ 2.60e-13 \\\hline \multirow{3}{*}{\parbox[c]{2.3cm}{duke\\(44,7129)}} & 1e-02  &  (82, 42) &  1.07 $|$ 0.17   &  2.53e-14 $|$ 1.50e-13 \\ & 1e-03  &  (337, 84) &  1.55 $|$ 0.32   &  9.66e-14 $|$ 4.75e-13 \\ & 1e-04  &  (1196, 339) &  1.85 $|$ 0.73   &  3.48e-08 $|$ 2.93e-12 \\\hline \multirow{3}{*}{\parbox[c]{2.3cm}{colon-cancer\\(62,2000)}} & 1e-02  &  (92, 59) &  0.29 $|$ 0.14   &  3.13e-14 $|$ 9.78e-13 \\ & 1e-03  &  (285, 100) &  0.47 $|$ 0.31   &  1.27e-13 $|$ 6.40e-12 \\ & 1e-04  &  (716, 288) &  0.77 $|$ 0.63   &  1.09e-08 $|$ 2.64e-11 \\\hline \multirow{3}{*}{\parbox[c]{2.3cm}{mpg7\\(392,3432)}} & 1e-02  &  (26, 19) &  1.61 $|$ 0.24   &  1.36e-12 $|$ 3.15e-10 \\ & 1e-03  &  (69, 51) &  1.64 $|$ 0.29   &  2.13e-12 $|$ 9.44e-10 \\ & 1e-04  &  (384, 192) &  2.91 $|$ 0.96   &  4.88e-12 $|$ 1.16e-10 \\\hline \multirow{3}{*}{\parbox[c]{2.3cm}{pyrim5\\(74,201376)}} & 1e-03  &  (314, 113) &  41.74 $|$ 5.57   &  9.68e-09 $|$ 1.25e-12 \\ & 1e-04  &  (2161, 433) &  42.30 $|$ 7.95   &  5.80e-08 $|$ 4.95e-12 \\ & 1e-05  &  (12191, 2304) &  47.28 $|$ 19.12   &  4.90e-08 $|$ 2.25e-08 \\\hline \multirow{3}{*}{\parbox[c]{2.3cm}{bodyfat7\\(252,116280)}} & 1e-03  &  (11, 4) &  33.68 $|$ 5.06   &  4.17e-14 $|$ 3.00e-12 \\ & 1e-04  &  (199, 14) &  54.26 $|$ 7.61   &  5.60e-09 $|$ 1.63e-12 \\ & 1e-05  &  (6901, 647) &  60.62 $|$ 15.01   &  6.78e-08 $|$ 7.48e-08 \\\hline \multirow{3}{*}{\parbox[c]{2.3cm}{triazines4\\(186,635376)}} & 1e-01  &  (269, 269) &  745.92 $|$ 106.51   &  6.43e-12 $|$ 1.28e-09 \\ & 1e-02  &  (439, 409) &  532.64 $|$ 94.63   &  6.60e-09 $|$ 5.66e-10 \\ & 1e-03  &  (1038, 699) &  487.09 $|$ 93.09   &  1.48e-09 $|$ 1.72e-09 \\\hline		
\end{tabular}}	
\end{table}	}

{\small\begin{table}[H]
\centering
\caption{The performances of DSSN(D) and PSSN(P) on selected  instances from UCI data repository.
}\label{Table:NEN2}\centering
\vskip 1.5mm
\begin{tabular}{ccccc}
\hline
problem name     & $\lambda$ &  nnz         &      time     & \multicolumn{1}{c}{error}\\
$(m,n) $;        &   ~   & ({\rm max},{\rm min})             &   $P$ $|$ $D$     & $P$ $|$ $D$\\
\hline
\multirow{3}{*}{\parbox[c]{2.3cm}{Leukaemia\\(38,7129)}} & 1e+02  &  (2411, 1221) &  2.79 $|$ 1.40  &  2.25e-09 $|$ 4.96e-09 \\ & 1e+01  &  (2185, 358) &  2.73 $|$ 0.87  &  1.18e-08 $|$ 2.05e-13 \\ & 1e+00  &  (1280, 94) &  2.96 $|$ 0.55  &  1.73e-08 $|$ 1.24e-13 \\ & 1e-01  &  (370, 44) &  3.53 $|$ 0.46  &  3.23e-13 $|$ 3.29e-13 \\ & 1e-02  &  (97, 39) &  6.91 $|$ 0.79  &  7.93e-09 $|$ 2.67e-12 \\\hline \multirow{3}{*}{\parbox[c]{2.3cm}{duke\\(44,7129)}} & 1e+02  &  (2151, 1131) &  3.28 $|$ 1.72  &  7.41e-09 $|$ 3.59e-13 \\ & 1e+01  &  (1972, 326) &  3.09 $|$ 1.11  &  2.56e-08 $|$ 2.30e-13 \\ & 1e+00  &  (1196, 82) &  3.61 $|$ 0.68  &  7.70e-09 $|$ 2.47e-13 \\ & 1e-01  &  (339, 42) &  3.72 $|$ 0.54  &  3.81e-09 $|$ 7.48e-13 \\ & 1e-02  &  (85, 39) &  5.85 $|$ 0.83  &  1.19e-10 $|$ 9.12e-12 \\\hline \multirow{3}{*}{\parbox[c]{2.3cm}{colon-cancer\\(62,2000)}} & 1e+02  &  (962, 640) &  1.12 $|$ 1.02  &  1.58e-08 $|$ 3.00e-13 \\ & 1e+01  &  (920, 254) &  1.00 $|$ 0.78  &  1.57e-08 $|$ 8.66e-13 \\ & 1e+00  &  (716, 92) &  0.89 $|$ 0.49  &  1.08e-13 $|$ 2.62e-12 \\ & 1e-01  &  (288, 59) &  1.02 $|$ 0.46  &  2.20e-13 $|$ 9.61e-12 \\ & 1e-02  &  (102, 55) &  1.97 $|$ 0.85  &  2.84e-13 $|$ 7.32e-11 \\\hline \multirow{3}{*}{\parbox[c]{2.3cm}{mpg7\\(392,3432)}} & 1e+02  &  (1574, 118) &  5.43 $|$ 1.95  &  4.64e-12 $|$ 5.70e-12 \\ & 1e+01  &  (914, 44) &  4.04 $|$ 0.88  &  4.03e-12 $|$ 1.58e-11 \\ & 1e+00  &  (384, 24) &  4.64 $|$ 0.83  &  3.62e-12 $|$ 3.54e-10 \\ & 1e-01  &  (192, 18) &  5.93 $|$ 1.06  &  3.34e-12 $|$ 6.39e-10 \\ & 1e-02  &  (138, 18) &  6.95 $|$ 1.23  &  3.28e-12 $|$ 9.39e-09 \\\hline \multirow{3}{*}{\parbox[c]{2.3cm}{pyrim5\\(74,201376)}} & 1e+02  &  (51298, 9319) &  78.46 $|$ 52.21  &  8.13e-08 $|$ 2.62e-10 \\ & 1e+01  &  (38448, 1613) &  69.98 $|$ 22.74  &  5.30e-08 $|$ 9.78e-09 \\ & 1e+00  &  (12191, 314) &  76.34 $|$ 15.20  &  3.26e-08 $|$ 8.37e-10 \\ & 1e-01  &  (2304, 113) &  104.40 $|$ 14.92  &  1.41e-08 $|$ 4.69e-12 \\ & 1e-02  &  (432, 71) &  328.09 $|$ 22.35  &  2.81e-08 $|$ 1.49e-11 \\\hline \multirow{3}{*}{\parbox[c]{2.3cm}{bodyfat7\\(252,116280)}} & 1e+02  &  (32516, 1663) &  79.46 $|$ 30.06  &  2.56e-07 $|$ 1.44e-08 \\ & 1e+01  &  (21871, 40) &  72.36 $|$ 16.28  &  5.42e-08 $|$ 4.87e-09 \\ & 1e+00  &  (6901, 11) &  55.50 $|$ 10.44  &  2.40e-09 $|$ 6.37e-11 \\ & 1e-01  &  (647, 4) &  57.23 $|$ 9.67  &  1.51e-09 $|$ 1.81e-11 \\ & 1e-02  &  (79, 4) &  92.08 $|$ 12.05  &  5.20e-14 $|$ 8.07e-10 \\\hline \multirow{3}{*}{\parbox[c]{2.3cm}{triazines4\\(186,635376)}} & 1e+02  &  (10693, 365) &  953.57 $|$ 134.71  &  9.17e-09 $|$ 4.84e-11 \\ & 1e+01  &  (2820, 295) &  764.55 $|$ 90.65  &  5.90e-10 $|$ 3.60e-10 \\ & 1e+00  &  (1038, 269) &  930.33 $|$ 111.89  &  5.33e-09 $|$ 2.85e-09 \\ & 1e-01  &  (699, 269) &  1399.66 $|$ 170.56  &  1.22e-07 $|$ 3.19e-08 \\ & 1e-02  &  (619, 269) &  6197.73 $|$ 442.13  &  3.46e-01 $|$ 2.58e-07 \\\hline
	
\end{tabular}	
\end{table}	}

\begin{figure}[H]
\centering
\includegraphics[scale=0.5]{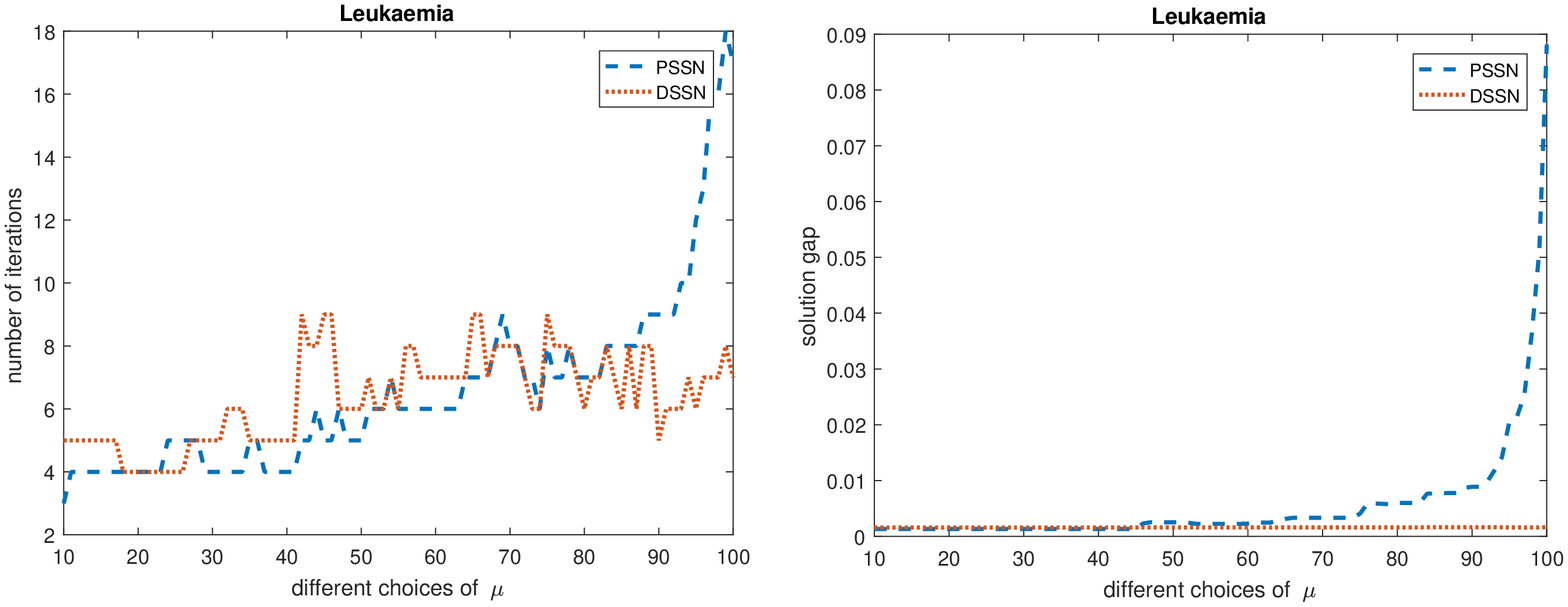}
%
\includegraphics[scale=0.5]{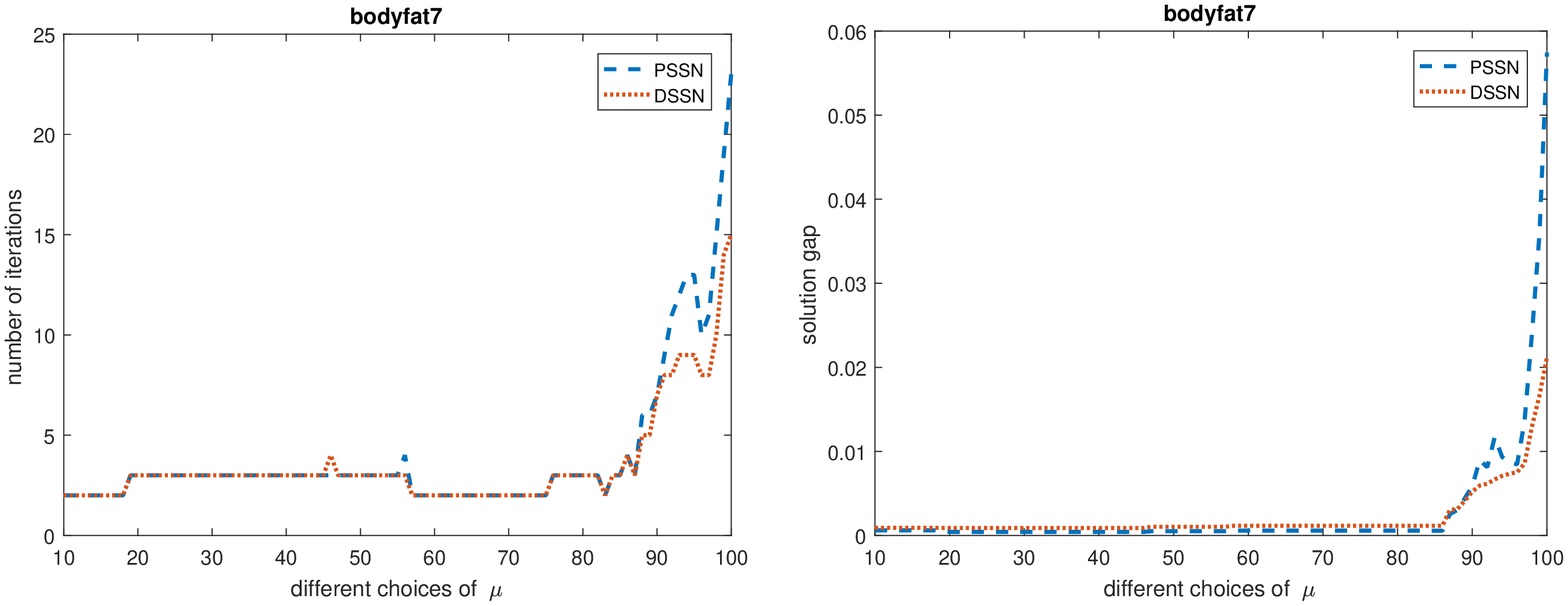}
\caption{Same as Figure \ref{Fig:ELasso}.}\label{Fig:NELasso}
\end{figure}

\section{Conclusion}\label{sec:conclusion}
In this paper, we have proposed a dual based semismooth Newton (DSSN) method for solving the sparse Tikhonov regularization. By taking advantage the Tikhonov regularization, the DSSN method can not only avoid increasing the scale of data matrix, but also can overcome the drawbacks of the primal semismooth Newton method (PSSN). In other words, the DSSN method is globally convergent and can achieve at least R-superlinearly convergence rate without reconstructing the data matrix.  The numerical efficiency and stability have also been demonstrated by comparing with the PSSN method and S{\footnotesize SNAL} on high-dimensional the UCI data sets. Finally, it is worth noting that this paper focuses on a class of simple Tikhonov regularization, i.e., the Tikhonov matrix is taken as $\alpha I ~(\alpha>0)$ . We leave the research topic that how to generalize DSSN to general Tikhonov matrix as our further work.


\begin{thebibliography}{10}

\bibitem{bai2017modulus}
Z.-Z. Bai, A.~Buccini, K.~Hayami, L.~Reichel, J.-F. Yin, and N.~Zheng.
\newblock Modulus-based iterative methods for constrained {Tikhonov}
  regularization.
\newblock {\em Journal of Computational and Applied Mathematics}, 319:1--13,
  2017.

\bibitem{beck2009fast}
A.~Beck and M.~Teboulle.
\newblock A fast iterative shrinkage-thresholding algorithm for linear inverse
  problems.
\newblock {\em SIAM Journal on Imaging Sciences}, 2(1):183--202, 2009.

\bibitem{chen2018generalized}
X.~Chen, A.~Y. Aravkin, and R.~D. Martin.
\newblock Generalized linear model for gamma distributed variables via elastic
  net regularization.
\newblock {\em arXiv preprint arXiv:1804.07780}, 2018.

\bibitem{clarke1990optimization}
F.~H. Clarke.
\newblock {\em Optimization and Nonsmooth Analysis}.
\newblock SIAM, 1990.

\bibitem{daducci2015accelerated}
A.~Daducci, E.~J. Canales-Rodr{\'\i}guez, H.~Zhang, T.~B. Dyrby, D.~C.
  Alexander, and J.-P. Thiran.
\newblock Accelerated microstructure imaging via convex optimization {(AMICO)}
  from diffusion {MRI} data.
\newblock {\em NeuroImage}, 105:32--44, 2015.

\bibitem{demirer2018estimating}
M.~Demirer, F.~X. Diebold, L.~Liu, and K.~Yilmaz.
\newblock Estimating global bank network connectedness.
\newblock {\em Journal of Applied Econometrics}, 33(1):1--15, 2018.

\bibitem{facchinei2007finite}
F.~Facchinei and J.-S. Pang.
\newblock {\em Finite-Dimensional Variational Inequalities and Complementarity
  Problems}.
\newblock Springer Science \& Business Media, 2007.

\bibitem{Facchinei2007}
F.~Facchinei and J.-S. Pang.
\newblock {\em Finite-dimensional variational inequalities and complementarity
  problems}.
\newblock Springer Science \& Business Media, 2007.

\bibitem{friedman2010regularization}
J.~Friedman, T.~Hastie, and R.~Tibshirani.
\newblock Regularization paths for generalized linear models via coordinate
  descent.
\newblock {\em Journal of Statistical Software}, 33(1):1, 2010.

\bibitem{gao2009calibrating}
Y.~Gao and D.~F. Sun.
\newblock Calibrating least squares covariance matrix problems with equality
  and inequality constraints.
\newblock {\em SIAM Journal on Matrix Analysis and Applications},
  31(3):1432--1457, 2009.

\bibitem{golub2012matrix}
G.~H. Golub and C.~F. Van~Loan.
\newblock {\em Matrix Computations (3rd ed.)}, volume~3.
\newblock John Hopkins University Press, 1996.

\bibitem{han2017linear}
D.~Han, D.~F. Sun, and L.~Zhang.
\newblock Linear rate convergence of the alternating direction method of
  multipliers for convex composite programming.
\newblock {\em Mathematics of Operations Research}, 43(2):622--637, 2017.

\bibitem{hiriart2013convex}
J.-B. Hiriart-Urruty and C.~Lemar{\'e}chal.
\newblock {\em Convex Analysis and Minimization Algorithms I: Fundamentals},
  volume 305.
\newblock Springer Science \& Business Media, 2013.

\bibitem{huang2010predicting}
L.~Huang, J.~Jia, B.~Yu, B.-G. Chun, P.~Maniatis, and M.~Naik.
\newblock Predicting execution time of computer programs using sparse
  polynomial regression.
\newblock In {\em Advances in Neural Information Processing Systems}, pages
  883--891, 2010.

\bibitem{lam2018fast}
X.~Y. Lam, J.~Marron, D.~F. Sun, and K.-C. Toh.
\newblock Fast algorithms for large-scale generalized distance weighted
  discrimination.
\newblock {\em Journal of Computational and Graphical Statistics},
  27(2):368--379, 2018.

\bibitem{lemarechal1997practical}
C.~Lemar{\'e}chal and C.~Sagastiz{\'a}bal.
\newblock Practical aspects of the {Moreau--Yosida} regularization: Theoretical
  preliminaries.
\newblock {\em SIAM Journal on Optimization}, 7(2):367--385, 1997.

\bibitem{li2018efficiently}
X.~Li, D.~Sun, and K.-C. Toh.
\newblock On efficiently solving the subproblems of a level-set method for
  fused lasso problems.
\newblock {\em SIAM Journal on Optimization}, 28(2):1842--1866, 2018.

\bibitem{li2016highly}
X.~Li, D.~F. Sun, and K.-C. Toh.
\newblock A highly efficient semismooth {Newton} augmented {Lagrangian} method
  for solving lasso problems.
\newblock {\em SIAM Journal on Optimization}, 28(2):1842--1866, 2018.

\bibitem{lichman2013uci}
M.~Lichman.
\newblock {UCI} machine learning repository.
\newblock {\em School of Information and Computer Sciences, University of
  California, Irvine}, 2013.

\bibitem{moreau1965proximite}
J.-J. Moreau.
\newblock Proximit{\'e} et dualit{\'e} dans un espace hilbertien.
\newblock {\em Bull. Soc. Math. France}, 93(2):273--299, 1965.

\bibitem{qi2006quadratically}
H.~Qi and D.~F. Sun.
\newblock A quadratically convergent {Newton} method for computing the nearest
  correlation matrix.
\newblock {\em SIAM Journal on Matrix Analysis and Applications},
  28(2):360--385, 2006.

\bibitem{qi1993convergence}
L.~Qi.
\newblock Convergence analysis of some algorithms for solving nonsmooth
  equations.
\newblock {\em Mathematics of Operations Research}, 18(1):227--244, 1993.

\bibitem{qi1993nonsmooth}
L.~Qi and J.~Sun.
\newblock {A nonsmooth version of Newton's method}.
\newblock {\em Mathematical Programming}, 58(1):353--367, 1993.

\bibitem{rockafellar1976augmented}
R.~T. Rockafellar.
\newblock Augmented {Lagrangians} and applications of the proximal point
  algorithm in convex programming.
\newblock {\em Mathematics of Operations Research}, 1(2):97--116, 1976.

\bibitem{rockafellar2015convex}
R.~T. Rockafellar.
\newblock {\em Convex Analysis}.
\newblock Princeton University Press, 2015.

\bibitem{sun2001solving}
D.~F. Sun and L.~Qi.
\newblock Solving variational inequality problems via smoothing-nonsmooth
  reformulations.
\newblock {\em Journal of Computational and Applied Mathematics},
  129(1-2):37--62, 2001.

\bibitem{teipel2017robust}
S.~J. Teipel et~al.
\newblock Robust detection of impaired resting state functional connectivity
  networks in {Alzheimer's} disease using elastic net regularized regression.
\newblock {\em Frontiers in Aging Neuroscience}, 8:318, 2017.

\bibitem{tibshirani1996regression}
R.~Tibshirani.
\newblock Regression shrinkage and selection via the lasso.
\newblock {\em Journal of the Royal Statistical Society. Series B
  (Methodological)}, pages 267--288, 1996.

\bibitem{tibshirani2012strong}
R.~Tibshirani, J.~Bien, J.~Friedman, T.~Hastie, N.~Simon, J.~Taylor, and R.~J.
  Tibshirani.
\newblock Strong rules for discarding predictors in lasso-type problems.
\newblock {\em Journal of the Royal Statistical Society: Series B (Statistical
  Methodology)}, 74(2):245--266, 2012.

\bibitem{wu2014nonnegative}
L.~Wu and Y.~Yang.
\newblock Nonnegative elastic net and application in index tracking.
\newblock {\em Applied Mathematics and Computation}, 227:541--552, 2014.

\bibitem{yosida1964functional}
K.~Yosida.
\newblock {\em Functional Analysis}.
\newblock Springer, Berlin, 1964.

\bibitem{Yu2013}
Y.-L. Yu.
\newblock On decomposing the proximal map.
\newblock In {\em Advances in Neural Information Processing Systems}, pages
  91--99, 2013.

\bibitem{Zhang2018efficient}
Y.~Zhang, N.~Zhang, D.~F. Sun, and K.-C. Toh.
\newblock An efficient {Hessian} based algorithm for solving large-scale sparse
  group {Lasso} problems.
\newblock{\em Mathematical Programming}, 179(1-2):223--263, 2020.

\bibitem{zhao2018semismooth}
H.-J. Zhao and H.~Yang.
\newblock Semismooth {Newton} methods with domain decomposition for {American}
  options.
\newblock {\em Journal of Computational and Applied Mathematics}, 337:37--50,
  2018.

\bibitem{zou2005regularization}
H.~Zou and T.~Hastie.
\newblock Regularization and variable selection via the elastic net.
\newblock {\em Journal of the Royal Statistical Society: Series B (Statistical
  Methodology)}, 67(2):301--320, 2005.

\end{thebibliography}
\end{document}